\documentclass[11pt]{article}

\usepackage[T1]{fontenc}
\usepackage[utf8]{inputenc}
\usepackage{amsmath,amssymb,amsthm,mathtools}
\usepackage[a4paper,margin=2.55cm]{geometry}
\usepackage{microtype}
\usepackage[colorlinks=true,linkcolor=blue,citecolor=blue,urlcolor=blue]{hyperref}
\hypersetup{
  pdftitle={Sharp Asymptotics for Abelian Covers of Groups with Bounded Noncommutativity},
  pdfauthor={Guillaume Lecomte}
}
\usepackage{enumitem}

\newtheorem{theorem}{Theorem}[section]
\newtheorem{lemma}[theorem]{Lemma}
\newtheorem{proposition}[theorem]{Proposition}
\newtheorem{corollary}[theorem]{Corollary}
\theoremstyle{definition}
\newtheorem{definition}[theorem]{Definition}
\theoremstyle{remark}
\newtheorem{remark}[theorem]{Remark}

\newcommand{\F}{\mathbb{F}}
\newcommand{\Aut}{\operatorname{Aut}}
\newcommand{\rad}{\operatorname{rad}}
\newcommand{\rk}{\operatorname{rk}}
\newcommand{\Tr}{\operatorname{Tr}}

\title{Sharp Asymptotics for Abelian Covers of Groups with Bounded Noncommutativity}
\author{Guillaume Lecomte\thanks{Corresponding author: \href{mailto:guillaume.lecomteexed@edu.executive.em-lyon.com}{guillaume.lecomteexed@edu.executive.em-lyon.com} (Paris, France).\\ORCID: 0009-0003-3532-3379.}}
\date{20 August 2026}

\begin{document}
\maketitle

\begin{abstract}
We determine the sharp exponential growth rate of the minimum number of abelian subgroups required to cover a group with bounded pairwise noncommutativity. Let $\omega(G)$ denote the largest size of a pairwise noncommuting subset of a group $G$, let $a(G)$ be the least size of an abelian cover, and define $h(n)$ as the supremum of $a(G)$ over groups satisfying $\omega(G)\le n$. Answering a quantitative question posed by Erd\H{o}s, we prove that
\[
\log_2 h(n)=\frac n2+O\!\left(\sqrt n\,(\log(n+2))^3\right),
\]
and hence that $h(n)^{1/n}$ tends to $\sqrt2$. Extraspecial $2$-groups provide the matching lower bound at the exponential scale. The upper bound combines a central-factor analysis of finite $p$-groups with alternating-form clique estimates, interaction control across central factors, Sylow decomposition, and a quasipolynomial-cost reduction through the centralizer of the derived subgroup. The argument also shows that asymptotic extremality is concentrated in $2$-groups and yields the same sharp exponential rate for the minimum possible index of an abelian subgroup. In graph-theoretic terms, the result determines the sharp asymptotic chromatic-versus-clique growth rate for noncommuting graphs of groups. This result resolves Erd\H{o}s Problem \#117 at the level of its sharp exponential asymptotics.
\end{abstract}



\section{Introduction}

A natural quantitative measure of noncommutativity is the largest cardinality $\omega(G)$ of a pairwise noncommuting subset of a group $G$. A complementary covering parameter is the least number $a(G)$ of abelian subgroups whose union is $G$. We study the extremal relation
\[
h(n)=\sup\{a(G):\omega(G)\le n\}.
\]
The quantitative covering question was posed by Erd\H{o}s in his collections of unsolved problems \cite{Erdos1990,Erdos1997}, but its graph-theoretic and group-theoretic antecedents go back considerably further. B.~H. Neumann proved that finiteness of a largest pairwise noncommuting set forces $G/Z(G)$ to be finite \cite{Neumann1954}, and later returned explicitly to Erd\H{o}s's question \cite{Neumann1976}. Erd\H{o}s and Straus \cite{ErdosStraus1976} investigated complementary quantitative measures of how far finite groups are from being abelian. Faber, Laver and McKenzie \cite{FaberLaverMcKenzie1978} formulated the associated noncommutation graph and studied its chromatic number in the language of coverings by abelian subgroups. Mason \cite{Mason1978} and Bertram \cite{Bertram1983} developed further connections between pairwise noncommuting sets, graph colouring and abelian covers. Pyber \cite{Pyber1987} then established the first uniform exponential bounds, proving that
\[
c_1^n<h(n)<c_2^n
\]
for absolute constants $c_2>c_1>1$. Erd\H{o}s \cite{Erdos1997} notes that the lower bound was already known to Isaacs. Thus the exponential scale was known, but its base was not determined.

The problem has an exact graph-theoretic formulation. The noncommuting graph $\Gamma_G$ has vertex set $G\setminus Z(G)$, with adjacency defined by noncommutation. Then $\omega(G)=\omega(\Gamma_G)$ and $a(G)=\max\{1,\chi(\Gamma_G)\}$. The later literature developed both sides of this correspondence. Brown determined asymptotic and exact separation phenomena for minimal abelian covers and maximal pairwise noncommuting sets in symmetric groups \cite{Brown1988,Brown1991}. Chin studied noncommuting sets in extraspecial $p$-groups \cite{Chin2005}. Abdollahi, Akbari and Maimani initiated a systematic study of noncommuting graphs \cite{AbdollahiAkbariMaimani2006}, while Azad and Praeger and subsequent authors obtained exact or structural results for particular linear and $p$-group families \cite{AzadPraeger2009,AzadIranmaneshPraegerSpiga2011,FouladiOrfi2011,FouladiOrfi2013,DarafshehGhorbaniPrajapati2015}. Quantitative relations between finite abelian coverings and derived groups were investigated by Podoski and Szegedy \cite{PodoskiSzegedy2002}. More general subgroup-covering questions form a substantial parallel literature \cite{Berkovich2010}, and recent work has also considered covers by centralizers \cite{LewisMcCulloch2026}.

These results are complementary to the present problem: they concern structural consequences, particular families, or other covering parameters. The question considered here is uniform over all groups and asks for the extremal dependence of the abelian covering number on the maximum size of a pairwise noncommuting set. The purpose of the present paper is to determine the exponential rate left open by the bounds of Pyber.

We prove the quantitative estimate
\[
\log_2 h(n)=\frac n2+O\!\left(\sqrt n\,(\log(n+2))^3\right),
\]
and hence $h(n)^{1/n}\to\sqrt2$. Extraspecial $2$-groups yield the lower bound. For the upper bound, isoclinism reduces the problem to finite groups. Two global estimates control conjugacy classes and the derived subgroup, using Pyber's 1987 BFC lemma and the bound of Neumann and Vaughan-Lee \cite{NeumannVaughanLee1977}.

It is perhaps worth noting that Pyber, in the concluding remarks of his 1987 paper, suggested that replacing the index of the centre by the index of a maximal abelian subgroup should lead to substantially sharper bounds, writing that ``probably $|G:A|\le 2^{4n}$ follows'' and leaving the details to the interested reader \cite{Pyber1987}. The present work may be viewed, in part, as carrying this suggestion to its asymptotic endpoint: at the level of abelian coverings, the optimal exponential constant is $1/2$.

The main new ingredient is the $p$-group analysis. A central series of $P'$ produces scalar alternating forms whose ranks measure the cost of a recursive abelian cover. Weak interaction between different levels is absorbed by a nested-anchor construction, whereas strong interaction forces a multiplicative noncommuting set after exact centralization. This gives coefficient $1/2$ for $p=2$ and strict asymptotic slack for odd primes. Sylow decomposition then handles nilpotent groups. Finally, for an arbitrary finite group $G$ we pass to the normal subgroup $H=C_G(G')$. Since $H'\le Z(H)$, this subgroup is nilpotent, while the Neumann--Vaughan-Lee estimate implies
\[
\log_2 [G:H]=O((\log N)^4).
\]
A domination argument on $H$-cosets then extends the estimate to $G$ at quasipolynomial cost.

\section{Statement and finite reduction}

This section fixes the graph-theoretic interpretation of the two parameters and removes the only infinitary issue in the problem. The key point is that both $\omega(G)$ and $a(G)$ depend only on the commutator structure modulo the centre. Isoclinism therefore permits a reduction to finite stem representatives without changing either quantity.

For an arbitrary group $G$, let $\Gamma_G$ be its noncommuting graph, with vertex set $G\setminus Z(G)$. Then $a(G)=\max\{1,\chi(\Gamma_G)\}$: for a nonabelian group, a commuting colour class generates an abelian subgroup, while an abelian cover gives a colouring; central elements may be inserted into every member of the cover. If $G$ is abelian, then $\Gamma_G$ has no vertices and $a(G)=1$. Thus the problem is a chromatic-versus-clique problem for the graph class $\Gamma_G$.

\begin{lemma}[Finite reduction by isoclinism]\label{lem:finite-reduction}
Let $G$ be a group with $\omega(G)<\infty$. Then there is a finite group $H$ isoclinic to $G$ such that
\[
\omega(H)=\omega(G),\qquad a(H)=a(G).
\]
Consequently,
\[
h(n)=\sup\{a(H):H\text{ finite and }\omega(H)\le n\}.
\]
\end{lemma}

\begin{proof}
By B.~H. Neumann's theorem \cite{Neumann1954}, finiteness of the largest pairwise noncommuting set implies that $G/Z(G)$ is finite. Hall's stem-group theorem \cite{Hall1940} (see also \cite[Proposition 3.2]{PournakiSobhani2008}) provides a group $H$ isoclinic to $G$ with $Z(H)\le H'$. Thus there are isomorphisms
\[
\alpha:G/Z(G)\longrightarrow H/Z(H),\qquad \beta:G'\longrightarrow H'
\]
that preserve commutators. Since $G/Z(G)$ is finite, Schur's theorem gives that $G'$ is finite; hence $H'$ is finite. As $Z(H)\le H'$, the centre $Z(H)$ is finite, while $H/Z(H)\cong G/Z(G)$ is finite. Therefore $H$ is finite.

It remains to check that both parameters are isoclinism invariants. Define a graph $\Delta_G$ on $G/Z(G)$ by joining two central cosets exactly when their representatives do not commute. This is well defined because multiplication by central elements does not change a commutator. A clique in $G$ contains at most one element from each central coset, and its image in $\Delta_G$ is a clique; conversely every clique of $\Delta_G$ lifts to a pairwise noncommuting subset of $G$. Thus $\omega(G)=\omega(\Delta_G)$.

For the covering number, enlarge every abelian subgroup in a cover to its product with $Z(G)$; it remains abelian. Such enlarged subgroups are unions of central cosets, and their images in $G/Z(G)$ are independent sets of $\Delta_G$. Hence an abelian cover gives a proper colouring of $\Delta_G$. Conversely, the union of the central cosets in one colour class consists of pairwise commuting elements, so the subgroup it generates is abelian. Therefore $a(G)=\chi(\Delta_G)$. The commutator-preserving isomorphism $\alpha$ is a graph isomorphism $\Delta_G\cong\Delta_H$, proving the claim.
\end{proof}

The following is the main result. It strengthens the previously known statement that $h(n)$ is merely exponential by determining the exact base of exponential growth.

\begin{theorem}\label{thm:main}
As $n\to\infty$,
\[
\log_2 h(n)=\frac n2+O\!\left(\sqrt n\,(\log(n+2))^3\right).
\]
In particular,
\[
\lim_{n\to\infty} h(n)^{1/n}=\sqrt2.
\]
\end{theorem}

\section{Global compression estimates}

Before entering the $p$-group analysis, we record two coarse but uniform consequences of bounded clique number. They control conjugacy classes and the derived subgroup. These losses are negligible compared with the exponential scale of the main theorem.

Throughout this section, put $N=\omega(G)$.

\begin{lemma}[Polynomial BFC bound]\label{lem:bfc}
Every conjugacy class of a finite group $G$ has size at most
\[
B_N:=(2N+1)^2.
\]
\end{lemma}

\begin{proof}
This is Pyber \cite[Lemma 3.1]{Pyber1987}, in the present notation.
\end{proof}

\begin{corollary}[Small derived subgroup]\label{cor:derived}
There is an absolute constant $C$ such that every finite group with $\omega(G)=N$ satisfies
\[
\log_2|G'|\le C\bigl(\log_2(N+2)\bigr)^2.
\]
\end{corollary}

\begin{proof}
By Lemma~\ref{lem:bfc}, $G$ is an $r$-BFC group with $r\le(2N+1)^2$. Neumann and Vaughan-Lee \cite{NeumannVaughanLee1977} proved the general bound
\[
|G'|\le r^{(3+5\log_2 r)/2}.
\]
Taking binary logarithms gives $\log_2|G'|=O((\log(N+2))^2)$, as required.
\end{proof}

\section{Symplectic toolkit}

The $p$-group recursion will convert commutators into alternating forms over finite fields. This section isolates the linear-algebraic ingredients required later: efficient covers by isotropic subspaces and lower bounds for pairwise nonorthogonal sets. The constants are chosen to make the eventual covering cost and clique credit directly comparable.

Let $V$ be a vector space over $\F_p$ and let $\phi$ be alternating of rank $\rho=2m$.

\begin{lemma}[Spread cover]\label{lem:spread}
The space $V$ is the union of at most $p^m+1$ $\phi$-isotropic subspaces.
\end{lemma}

\begin{proof}
If $m=0$, then $\phi=0$ and $V$ itself is isotropic, so the assertion is immediate. Assume henceforth that $m\ge1$. It is enough to treat the nondegenerate quotient $W=V/\rad\phi$, which has dimension $2m$. Put $E=\F_{p^m}$ and identify $W$ with $E^2$ as an $\F_p$-space. Up to isometry we may use
\[
\langle(x,y),(x',y')\rangle=\Tr_{E/\F_p}(xy'-x'y).
\]
For $t\in E$ let
\[
L_t=\{(x,tx):x\in E\},\qquad L_\infty=\{(0,y):y\in E\}.
\]
Each is $m$-dimensional and totally isotropic. Every nonzero vector $(x,y)$ lies in exactly one of them: in $L_\infty$ if $x=0$, and otherwise in $L_{y/x}$. Their inverse images in $V$ are $\phi$-isotropic and cover $V$.
\end{proof}

\begin{definition}\label{def:kappa}
Set
\[
\kappa_2=1,\qquad \kappa_3=2,\qquad \kappa_p=\frac p2\quad(p\ge5),
\]
and $c_2=0$, $c_3=2$, $c_p=0$ for $p\ge5$.
\end{definition}

\begin{lemma}[Orthogonal-anchor composition]\label{lem:anchor}
Let $(V_1,\phi_1)$ and $(V_2,\phi_2)$ be alternating spaces and let $C_i\subseteq V_i$ be pairwise nonorthogonal sets with distinguished vertices $u_i\in C_i$. In $V_1\perp V_2$, the set
\[
C=(C_1\setminus\{u_1\})\cup\{u_1+v:v\in C_2\}
\]
is pairwise nonorthogonal, has cardinality $|C|=|C_1|+|C_2|-1$, and may be given distinguished vertex $u_1+u_2$.
\end{lemma}

\begin{proof}
Two vertices $u_1+v,u_1+w$ in the second block have pairing $\phi_2(v,w)\ne0$. If $x\in C_1\setminus\{u_1\}$, then $(\phi_1\perp\phi_2)(x,u_1+v)=\phi_1(x,u_1)\ne0$. The first block is a clique by assumption, and the two displayed parts lie in disjoint affine subspaces.
\end{proof}

\begin{lemma}[Scalar clique credit]\label{lem:scalar-credit}
For every alternating form of rank $\rho$ over $\F_p$, there is a pairwise nonorthogonal set $C$ satisfying
\[
|C|-1\ge \kappa_p\rho-c_p.
\]
\end{lemma}

\begin{proof}
If $\rho=0$, take $C$ to be any singleton. Assume henceforth that $\rho>0$. For $p=2$, a nondegenerate rank-$\rho=2m$ symplectic space has a clique of size $2m+1=\rho+1$: start from $\{e,f,e+f\}$ in one hyperbolic plane and use Lemma~\ref{lem:anchor} inductively.

For odd $p$, a hyperbolic plane contains
\[
\{e+\lambda f:\lambda\in\F_p\}\cup\{f\},
\]
a clique of size $p+1$. Repeated application of Lemma~\ref{lem:anchor} gives a clique of size $pm+1=(p/2)\rho+1$, proving the claim for $p\ge5$.

For $p=3$, consider the following $13$ projective points in $\F_3^6$:
\begin{align*}
&(1,2,0,2,1,1),(0,1,1,0,0,2),(0,1,0,1,0,0),(0,1,1,2,2,1),\\
&(1,2,1,0,2,0),(1,1,0,1,2,2),(1,0,0,1,2,0),(1,0,1,1,2,1),\\
&(1,1,1,0,0,1),(1,2,0,0,2,1),(1,0,0,1,1,1),(1,0,1,1,1,1),\\
&(1,2,0,1,1,0).
\end{align*}
They are pairwise nonorthogonal for
\[
x_1y_2-x_2y_1+x_3y_4-x_4y_3+x_5y_6-x_6y_5.
\]
Indeed, the entries strictly above the diagonal of the corresponding $13\times13$ symplectic Gram matrix, read row by row, are
\begin{align*}
&1,1,1,2,2,2,1,1,2,1,2,2;\\
&1,1,1,2,2,2,2,1,1,1,1;\\
&2,1,2,2,1,1,2,2,1,2;\\
&1,2,1,1,2,2,1,2,2;\\
&1,2,1,1,2,1,1,1;\\
&1,2,1,2,2,1,2;\\
&1,2,1,2,1,2;\\
&2,2,2,1,2;\\
&2,2,2,1;\\
&2,2,2;\\
&2,1;\\
&2,
\end{align*}
so none of the $78$ pairings vanishes. Decompose a rank-$2m$ space into $q=\lfloor m/3\rfloor$ rank-six blocks and at most two hyperbolic planes. Applying Lemma~\ref{lem:anchor} block by block gives clique credit
\[
12q+3(m-3q)\ge4m-2=2\rho-2.
\]
\end{proof}

\begin{remark}\label{rem:alpha}
For every prime $p$,
\[
\alpha_p:=\frac{\log_2 p}{2\kappa_p}\le\frac12,
\]
with equality only at $p=2$. Explicitly,
\[
\alpha_3=\frac{\log_2 3}{4}<0.397,\qquad \alpha_p=\frac{\log_2 p}{p}<\frac12\quad(p\ge5).
\]
\end{remark}

\subsection{The extraspecial lower bound}

Extraspecial groups are classical extremal examples in the study of pairwise noncommuting sets; see, for example, \cite{Chin2005} and the discussion around the lower bound attributed to Isaacs in \cite{Erdos1997,Pyber1987}. We include the short covering argument because the precise value of both parameters is what fixes the sharp exponential constant.

\begin{lemma}[Extraspecial lower bound]\label{lem:extraspecial}
For every $m\ge1$ there is an extraspecial $2$-group $E_m$ such that
\[
\omega(E_m)=2m+1,\qquad a(E_m)=2^m+1.
\]
Consequently,
\[
\log_2 h(n)\ge\frac n2-O(1).
\]
\end{lemma}

\begin{proof}
Let $E_m$ be an extraspecial $2$-group of order $2^{1+2m}$. Then $V=E_m/Z(E_m)$ is a $2m$-dimensional symplectic space over $\F_2$, with
\[
\langle xZ(E_m),yZ(E_m)\rangle=[x,y]\in Z(E_m)\cong\F_2.
\]
Thus pairwise noncommuting subsets of $E_m$ correspond to pairwise nonorthogonal subsets of $V$.

Let $C=\{v_1,\ldots,v_t\}$ be such a subset. Its Gram matrix is $J_t+I_t$ over $\F_2$, whose rank is $t$ when $t$ is even and $t-1$ when $t$ is odd. Since this Gram matrix factors through the ambient symplectic form, its rank is at most $2m$, so $t\le2m+1$. Equality is attained by repeated use of Lemma~\ref{lem:anchor}. Hence $\omega(E_m)=2m+1$.

A symplectic spread of $V$ has $2^m+1$ Lagrangian members and gives an abelian cover of $E_m$ of that size. Conversely, in any abelian cover we may replace each abelian subgroup $A$ by $AZ(E_m)$; this remains abelian and its image $AZ(E_m)/Z(E_m)$ is an isotropic subspace of $V$. Every such isotropic subspace contains at most $2^m-1$ nonzero vectors, so a cover of the $2^{2m}-1$ nonzero vectors requires at least
\[
\frac{2^{2m}-1}{2^m-1}=2^m+1
\]
members. Thus $a(E_m)=2^m+1$. For arbitrary $n$, take $m=\lfloor(n-1)/2\rfloor$; then $\omega(E_m)\le n$ and
\[
\log_2(2^m+1)=\frac n2-O(1).
\]
\end{proof}

\section{Central-factor descent in a finite \texorpdfstring{$p$}{p}-group}

This is the structural core of the proof. We descend through a central series of the derived subgroup and, at each level, replace the surviving commutator by a scalar alternating form. The resulting cover tree is efficient only if the interaction between different levels is controlled; the remainder of the section develops the two mechanisms that do so.

Let $P$ be a finite $p$-group and put $n=\omega(P)$.

\begin{lemma}[Central-factor series]\label{lem:central-series}
There is a normal series
\begin{equation}\label{eq:central-series}
1=K_0<K_1<\cdots<K_L=P'
\end{equation}
such that
\[
|K_i:K_{i-1}|=p,\qquad K_i/K_{i-1}\le Z(P/K_{i-1})
\]
for every $i$.
\end{lemma}

\begin{proof}
If $K_i<P'$, then $P'/K_i$ is a nontrivial normal subgroup of the finite $p$-group $P/K_i$, so it meets $Z(P/K_i)$ nontrivially. That intersection contains a subgroup of order $p$; its inverse image defines $K_{i+1}$. Iteration terminates at $P'$.
\end{proof}

\subsection{The recursive cover}

Fix a branch of the recursion and set $A_0=P$.

\begin{lemma}[Central-factor form and descent]\label{lem:descent}
Suppose $0\le j<L$ and $A_j\le P$ satisfies $[A_j,A_j]\le K_{L-j}$. Define
\[
R_j=\{x\in A_j:[x,A_j]\le K_{L-j-1}\}.
\]
Then $V_j=A_j/R_j$ is an elementary abelian $p$-group, and commutation modulo $K_{L-j-1}$ induces a nondegenerate alternating form
\[
\phi_j:V_j\times V_j\longrightarrow K_{L-j}/K_{L-j-1}\cong\F_p.
\]
Writing $\rho_j=\rk\phi_j=\dim_{\F_p}V_j$, the group $A_j$ is covered by at most $p^{\rho_j/2}+1$ subgroups $A_{j+1}$ satisfying $[A_{j+1},A_{j+1}]\le K_{L-j-1}$.
\end{lemma}

\begin{proof}
In the quotient $\overline A_j=A_j/(A_j\cap K_{L-j-1})$, the derived subgroup is contained in the central subgroup $K_{L-j}/K_{L-j-1}$, so $\overline A_j$ has nilpotency class at most two. Moreover $R_j/(A_j\cap K_{L-j-1})=Z(\overline A_j)$, hence $R_j\trianglelefteq A_j$. Because $[A_j,A_j]\le K_{L-j}$ and $K_{L-j}/K_{L-j-1}$ is central, every commutator lies in $R_j$, so $A_j/R_j$ is abelian.

With the convention $[x,y]=x^{-1}y^{-1}xy$, class-two commutator calculus gives
\[
[x^p,y]\equiv[x,y]^p\equiv1\pmod{K_{L-j-1}},
\]
so $x^p\in R_j$. Thus $V_j$ is an $\F_p$-vector space. Define
\[
\phi_j(xR_j,yR_j)=[x,y]K_{L-j-1}.
\]
The class-two identities show that this is well defined, bilinear and alternating; its radical is zero by the definition of $R_j$. The covering statement follows from Lemma~\ref{lem:spread}.
\end{proof}

Repeating Lemma~\ref{lem:descent} produces a finite rooted cover tree. At depth $L$, every branch subgroup has trivial derived subgroup and is therefore abelian.

\begin{lemma}[Branch cover bound]\label{lem:branch}
The recursive cover satisfies
\[
\log_2 a(P)\le \max_B\left(\frac{\log_2 p}{2}\sum_{j=0}^{L-1}\rho_j(B)+L\right),
\]
where the maximum is over root-to-leaf branches $B$ and $\rho_j(B)$ denotes the rank encountered at stage $j$ on that branch.
\end{lemma}

\begin{proof}
At each node the number of children is at most $p^{\rho_j/2}+1$, and
\[
\log_2(p^{\rho_j/2}+1)\le\frac{\rho_j}{2}\log_2p+1.
\]
Induction from the leaves bounds the total number of leaves by the largest product of the branching bounds along a root-to-leaf branch. Taking logarithms gives the claim.
\end{proof}

\subsection{Interaction ranks}

For $j<k$, the image of $A_k$ in $V_j$ is $\phi_j$-isotropic. Define
\[
\psi_{j,k}:A_k\longrightarrow V_j^*,\qquad y\longmapsto\bigl(xR_j\longmapsto\phi_j(xR_j,yR_j)\bigr),
\]
Since commutators are central modulo $K_{L-j-1}$, the class-two commutator identities show that $\psi_{j,k}$ is a group homomorphism from $A_k$ to the additive group $V_j^*$. Put $t_{j,k}=\rk\psi_{j,k}$. Thus $\ker\psi_{j,k}$ is a subgroup of $A_k$ and, by the first isomorphism theorem,
\[
[A_k:\ker\psi_{j,k}]=|\operatorname{im}\psi_{j,k}|=p^{t_{j,k}}.
\]
Since $\phi_j$ is nondegenerate, $t_{j,k}$ is the dimension of the image of $A_k$ in $V_j$.

\begin{lemma}[Transversal interaction clique]\label{lem:transversal}
For every $j<k$ and every $1\le d\le t_{j,k}$ there is a clique
\[
T=\{a_0,\ldots,a_d\}\subseteq A_j
\]
such that the $a_i$ lie in pairwise distinct left cosets of $A_k$, and for every $r\ne s$ the commutator $[a_r,a_s]$ has nonzero image in $K_{L-j}/K_{L-j-1}$.
\end{lemma}

\begin{proof}
Choose $y_1,\ldots,y_d\in A_k$ whose images span a $d$-dimensional isotropic subspace $Y\le V_j$. Extend $Y$ to hyperbolic pairs $(x_i,y_i)$ with
\[
\phi_j(x_i,y_h)=\delta_{ih},\qquad \phi_j(x_i,x_h)=0.
\]
In $V_j$ set
\[
a_0=y_1+\cdots+y_d,\qquad a_i=x_i+\sum_{h<i}y_h\quad(1\le i\le d).
\]
For $i<r$, one has $\phi_j(a_i,a_r)=1$, while $\phi_j(a_0,a_i)=-1$, so the set is a clique. Let $U$ be the full image of $A_k$ in $V_j$. Since $U$ is isotropic and $Y\le U$, no nonzero linear combination of the $x_i$ lies in $U$: pairing with $y_h$ recovers its $h$th coefficient. Hence the $x_i$ are linearly independent modulo $U$. Since all added $y_h$ lie in $U$, the displayed $a_i$ lie in pairwise distinct $U$-cosets. Lift them to $A_j$.
\end{proof}

\subsection{Weak interaction: nested-anchor composition}

Let $B$ be an upper bound for all conjugacy class sizes in $P$, and put $\ell=\lceil\log_pB\rceil$.

\begin{lemma}[Nested-anchor composition]\label{lem:nested}
For every branch of the central-factor recursion,
\begin{equation}\label{eq:nested}
n\ge \kappa_p\sum_{k=0}^{L-1}\rho_k-2\kappa_p\sum_{j<k}t_{j,k}-O\!\left(\kappa_pL^2\ell+c_pL\right).
\end{equation}
\end{lemma}

\begin{proof}
We recursively build a clique $S_k$ using the first $k+1$ stages and a distinguished element $z_k$. At stage $k$, restrict $A_k$ to
\[
H_k=A_k\cap\bigcap_{j<k}\ker\psi_{j,k}\cap\bigcap_{j<k}C_P(a_j),
\]
where the $a_j$ are the previously chosen stage anchors. By the preceding observation, each kernel $\ker\psi_{j,k}$ has index $p^{t_{j,k}}$ in $A_k$, while
\[
[A_k:A_k\cap C_P(a_j)]\le[P:C_P(a_j)]\le B.
\]
Therefore
\[
[A_k:H_k]\le p^{\sum_{j<k}t_{j,k}}B^k\le p^{q_k},\qquad q_k:=\sum_{j<k}t_{j,k}+k\ell.
\]
Passing to the image in $V_k$ cannot increase this index, so that image has codimension at most $q_k$. If an alternating form of rank $\rho$ is restricted to a subspace of codimension $q$, its radical can grow by at most $q$ and the ambient dimension drops by $q$; hence its rank drops by at most $2q$. By Lemma~\ref{lem:scalar-credit}, the image of $H_k$ in $V_k$ contains a pairwise nonorthogonal set of cardinality at least $\kappa_p(\rho_k-2q_k)-c_p+1$. Choose one representative in $H_k$ of each vector in this set; the resulting stage clique $C_k\subseteq H_k$ may be given a distinguished element $a_k\in C_k$ and satisfies
\[
|C_k|-1\ge\kappa_p(\rho_k-2q_k)-c_p.
\]
Set $S_0=C_0$, $z_0=a_0$, and inductively
\[
S_k=(S_{k-1}\setminus\{z_{k-1}\})\cup\{z_{k-1}c:c\in C_k\},\qquad z_k=z_{k-1}a_k.
\]
Every element of $C_k$ centralizes all previous anchors and hence $z_{k-1}$, so the new block is internally a clique.

Let $x\in S_{k-1}\setminus\{z_{k-1}\}$ have birth stage $r<k$. Then $x=z_{r-1}c_r$ with $c_r\in C_r\setminus\{a_r\}$, with the evident interpretation when $r=0$. Since all elements of $C_r$ centralize earlier anchors,
\[
[x,z_r]=[c_r,a_r]\not\equiv1\pmod{K_{L-r-1}}.
\]
For $s>r$, the anchor $a_s$ centralizes all earlier anchors and belongs to $\ker\psi_{r,s}$, so $[c_r,a_s]\in K_{L-r-1}$. Using $[x,yz]=[x,z][x,y]^z$ and the centrality of $K_{L-r}/K_{L-r-1}$ in $P/K_{L-r-1}$, induction gives
\[
[x,z_s]\equiv[c_r,a_r]\not\equiv1\pmod{K_{L-r-1}}.
\]
The same calculation with $c\in C_k$ shows
\[
[x,z_{k-1}c]\equiv[c_r,a_r]\not\equiv1\pmod{K_{L-r-1}}.
\]
Thus every old non-distinguished vertex is noncommuting with every vertex of the new block. The two parts of the union are disjoint because every new element commutes with $z_{k-1}$, while every old non-distinguished element does not. Therefore $S_k$ is a clique and
\[
|S_{L-1}|=1+\sum_k(|C_k|-1).
\]
Summing the stage-credit estimates gives \eqref{eq:nested}.
\end{proof}

\begin{corollary}[Selected-stage composition]\label{cor:selected}
For every subset $E\subseteq\{0,\ldots,L-1\}$,
\[
n\ge\kappa_p\sum_{k\in E}\rho_k-2\kappa_p\!\sum_{\substack{j<k\\ j,k\in E}}t_{j,k}-O\!\left(\kappa_p|E|^2\ell+c_p|E|\right).
\]
\end{corollary}

\begin{proof}
Run the construction of Lemma~\ref{lem:nested} only at indices in $E$. At a selected stage $k$, intersect only the kernels and centralizers corresponding to earlier selected indices. The birth-stage argument is unchanged.
\end{proof}

\subsection{Large interaction: exact centralization}

\begin{lemma}[Interaction product inequality]\label{lem:product}
Fix $j<k$. For every integer $1\le d\le t_{j,k}$,
\[
n\ge(d+1)\bigl(\kappa_p(\rho_k-2(d+1)\ell)-c_p+1\bigr)
\]
whenever the factor in parentheses is positive.
\end{lemma}

\begin{proof}
Take the transversal clique $T=\{a_0,\ldots,a_d\}$ from Lemma~\ref{lem:transversal} and define
\[
H=A_k\cap\bigcap_{a\in T}C_P(a).
\]
Then $[A_k:H]\le B^{d+1}$, so the stage-$k$ form restricted to the image of $H$ in $V_k$ has rank at least $\rho_k-2(d+1)\ell$. By Lemma~\ref{lem:scalar-credit}, that image contains a pairwise nonorthogonal set of cardinality at least
\[
\kappa_p(\rho_k-2(d+1)\ell)-c_p+1.
\]
Choose one representative in $H$ of each vector in this set, and denote the resulting set of representatives by $D$. Then $D\subseteq H$ is a clique and
\[
|D|\ge\kappa_p(\rho_k-2(d+1)\ell)-c_p+1.
\]
All elements of $D$ centralize all elements of $T$. The product set $TD=\{ad:a\in T,d\in D\}$ has cardinality $(d+1)|D|$, because the elements of $T$ occupy distinct $A_k$-cosets. If $a\ne a'$, exact centralization gives
\[
[ad,a'd']=[a,a'][d,d'].
\]
Here $[a,a']\in K_{L-j}\setminus K_{L-j-1}$, whereas $[d,d']\in[A_k,A_k]\le K_{L-k}\le K_{L-j-1}$, so $[ad,a'd']\ne1$. If $a=a'$ and $d\ne d'$, then $[ad,ad']=[d,d']\ne1$. Hence $TD$ is a clique.
\end{proof}

\begin{corollary}[Expensive stages have small interaction]\label{cor:small-interaction}
There is an absolute constant $C_0$ such that, whenever
\[
\kappa_p\rho_k^2\ge C_0n\ell,
\]
one has, for every $j<k$,
\[
t_{j,k}=O\!\left(\frac{n}{\kappa_p\rho_k}\right),
\]
with an absolute implied constant.
\end{corollary}

\begin{proof}
Put $d=\lceil4n/(\kappa_p\rho_k)\rceil$ and suppose $t_{j,k}\ge d$. The scalar-credit lemma applied to the full stage-$k$ form gives $n\ge\kappa_p\rho_k-c_p+1$. Since $\rho_k\ge2$, this implies uniformly that $n\ge(3/4)\kappa_p\rho_k$: for $p=2$ and $p\ge5$ the stronger inequality $n\ge\kappa_p\rho_k+1$ holds, while for $p=3$ one has $n\ge2\rho_k-1\ge(3/2)\rho_k=(3/4)\kappa_3\rho_k$. Hence the hypothesis gives $\rho_k\ge(3C_0/4)\ell$. Also
\[
2(d+1)\ell\le\frac{8n\ell}{\kappa_p\rho_k}+4\ell
\le\left(\frac8{C_0}+\frac{16}{3C_0}\right)\rho_k.
\]
Taking $C_0=100$ makes this at most $\rho_k/3$. Therefore
\[
\kappa_p(\rho_k-2(d+1)\ell)-c_p+1\ge\frac{\kappa_p\rho_k}{3}.
\]
By Lemma~\ref{lem:product},
\[
n\ge(d+1)\frac{\kappa_p\rho_k}{3}>\frac{4n}{3},
\]
a contradiction. Hence $t_{j,k}<d$.
\end{proof}

\section{The \texorpdfstring{$p$}{p}-group upper bound}

We now assemble the cover-tree estimate with the interaction bounds. The argument separates stages according to whether their symplectic rank is large enough to matter on the exponential scale. Expensive stages have tightly controlled cross-interaction, while cheap stages contribute only a lower-order error. This yields a uniform bound in the prime and identifies $p=2$ as the unique asymptotically extremal case.

The next theorem is the quantitative heart of the proof. Its prime-dependent coefficient records the ratio between logarithmic covering cost and the amount of clique forced by one unit of symplectic rank.

\begin{theorem}[Uniform $p$-group bound]\label{thm:pgroup}
Let $P$ range over finite nonabelian $p$-groups and put $n=\omega(P)$. Uniformly in the prime $p$,
\[
\log_2a(P)\le\alpha_p n+O\!\left(\frac{n}{\log(n+2)}\right),
\]
where
\[
\alpha_2=\frac12,\qquad \alpha_3=\frac{\log_2 3}{4},\qquad \alpha_p=\frac{\log_2 p}{p}\quad(p\ge5).
\]
In particular,
\[
\log_2a(P)\le\frac n2+O\!\left(\frac{n}{\log(n+2)}\right),
\]
with strict asymptotic slack for every odd prime.
\end{theorem}

\begin{proof}
Fix an arbitrary root-to-leaf branch of the recursive cover, and use the associated quantities $\rho_k$ and $t_{j,k}$. All estimates below are uniform in the chosen branch, so Lemma~\ref{lem:branch} may be applied at the end. By Lemma~\ref{lem:bfc}, $P$ is a $B$-BFC group with $B=O(n^2)$. By Corollary~\ref{cor:derived},
\[
\log_2|P'|=O((\log n)^2),
\]
so the series \eqref{eq:central-series} has length
\[
L=O\!\left(\frac{(\log n)^2}{\log p}\right).
\]

First suppose $p\le(\log n)^3$. Put $Q=(\log n)^{10}$ and call a stage expensive if $\kappa_p\rho_k\ge n/Q$. The total weighted rank of cheap stages is at most
\[
L\frac nQ=O\!\left(\frac{n}{(\log n)^8}\right).
\]
For an expensive stage, $\rho_k\ge n/(\kappa_pQ)$. Since $p\le(\log n)^3$, we have $\kappa_p=O((\log n)^3)$ and $\ell=O(\log n)$. Consequently
\[
\frac{\kappa_p\rho_k^2}{n\ell}\ge\frac{n}{\kappa_pQ^2\ell}\longrightarrow\infty,
\]
because $\kappa_pQ^2\ell=O((\log n)^{24})$. Thus, for all sufficiently large $n$, the hypothesis of Corollary~\ref{cor:small-interaction} holds uniformly in $p$ and in the branch. Thus $t_{j,k}=O(Q)$ for every interaction entering an expensive stage. There are $O(L^2)$ pairs, so
\[
\kappa_p\sum_{\substack{j<k\\ k\text{ expensive}}}t_{j,k}=o(n).
\]
The anchor loss $O(\kappa_pL^2\ell+c_pL)$ is also $o(n)$. Applying Corollary~\ref{cor:selected} to the expensive stages gives
\[
\sum_{k\text{ expensive}}\kappa_p\rho_k
\le n+O\!\left(\frac{n}{(\log n)^8}+(\log n)^{17}\right).
\]
Adding the cheap stages gives the same bound for $\sum_k\kappa_p\rho_k$. By Lemma~\ref{lem:branch},
\[
\log_2a(P)\le\frac{\log_2p}{2\kappa_p}
\left(n+O\!\left(\frac{n}{(\log n)^8}+(\log n)^{17}\right)\right)+O(L)
=\alpha_p n+O\!\left(\frac{n}{\log(n+2)}\right).
\]

Now suppose $p>(\log n)^3$. Since $P$ is nonabelian, some stage has positive rank; Lemma~\ref{lem:scalar-credit} then gives a clique of size at least $p+1$, so $p<n$. Every positive-rank stage has an individual clique of size at least $(p/2)\rho_k+1$, hence $\rho_k\le2(n-1)/p$. Therefore
\[
\log_2a(P)\le O\!\left(L\frac{n\log p}{p}\right)
=O\!\left(\frac{n(\log n)^2}{p}\right)
=O\!\left(\frac{n}{\log n}\right).
\]
All implied constants are absolute.
\end{proof}

\subsection{Quantitative optimization in the critical 2-group case}

The preceding proof was organized only to obtain a uniform $o(n)$ error. In the critical case $p=2$, the same structural lemmas give a substantially sharper estimate after optimizing the threshold separating cheap and expensive stages.

\begin{proposition}[Quantitative 2-group bound]\label{prop:2group}
Let $P$ be a finite $2$-group and put $n=\omega(P)$. Then
\[
\log_2a(P)\le\frac n2+O\!\left(\sqrt n\,(\log(n+2))^3\right),
\]
with an absolute implied constant.
\end{proposition}

\begin{proof}
Fix a root-to-leaf branch of the recursive cover and retain the notation $\rho_k,t_{j,k},L$ and $\ell$ from Section~5. By Lemma~\ref{lem:branch},
\begin{equation}\label{eq:61}
\log_2a(P)\le\frac12\sum_{k=0}^{L-1}\rho_k+L.
\end{equation}
By Corollary~\ref{cor:derived}, $L=O((\log(n+2))^2)$, while Lemma~\ref{lem:bfc} gives $\ell=O(\log(n+2))$.

Choose
\[
R=A\sqrt{n(L+\ell)},
\]
where $A$ is a sufficiently large absolute constant, and let $E=\{k:\rho_k\ge R\}$. The complementary stages contribute at most
\begin{equation}\label{eq:62}
\sum_{k\notin E}\rho_k\le LR.
\end{equation}
For $k\in E$, the choice of $A$ ensures $\rho_k^2\ge C_0n\ell$, so Corollary~\ref{cor:small-interaction} applies and gives, uniformly for $j<k$,
\begin{equation}\label{eq:63}
t_{j,k}=O\!\left(\frac n{\rho_k}\right)=O\!\left(\frac nR\right).
\end{equation}
Apply Corollary~\ref{cor:selected} only to the selected stages $E$. Since $\kappa_2=1$ and $c_2=0$,
\[
n\ge\sum_{k\in E}\rho_k-2\!\sum_{\substack{j<k\\j,k\in E}}t_{j,k}-O(L^2\ell).
\]
There are $O(L^2)$ selected pairs; hence \eqref{eq:63} yields
\begin{equation}\label{eq:64}
\sum_{k\in E}\rho_k\le n+O\!\left(\frac{L^2n}{R}+L^2\ell\right).
\end{equation}
Combining \eqref{eq:62} and \eqref{eq:64},
\begin{equation}\label{eq:65}
\sum_k\rho_k\le n+O\!\left(LR+\frac{L^2n}{R}+L^2\ell\right).
\end{equation}
With $R=A\sqrt{n(L+\ell)}$, $L=O((\log(n+2))^2)$ and $\ell=O(\log(n+2))$, every error term in \eqref{eq:65} is
\[
O\!\left(\sqrt n\,(\log(n+2))^3\right);
\]
the purely polylogarithmic term $L^2\ell$ is smaller for large $n$. Substitution in \eqref{eq:61} proves the proposition.
\end{proof}

The square-root scale in Proposition~\ref{prop:2group} reflects the present pairwise interaction estimate $t_{j,k}=O(n/\rho_k)$: the losses $LR$ and $L^2n/R$ balance at square-root order.

\section{Nilpotent groups}

The passage from $p$-groups to nilpotent groups uses the direct-product structure of Sylow subgroups. Taking products of abelian covers gives a multiplicative upper bound for the covering number, whereas Cartesian products of pairwise noncommuting sets give a multiplicative lower bound for the clique number. This mismatch is favourable: if more than one Sylow subgroup is nonabelian, the resulting estimate has substantial slack.

The product structure now turns the $p$-group estimates into a nilpotent estimate. The second conclusion is useful conceptually: asymptotic extremality can only come from groups with essentially one nonabelian Sylow component.

\begin{theorem}[Nilpotent upper bound]\label{thm:nilpotent}
Uniformly for finite nilpotent groups $F$ with $n=\omega(F)$,
\[
\log_2a(F)\le\frac n2+o(n).
\]
If at least two Sylow subgroups of $F$ are nonabelian, then
\[
\log_2a(F)\le\frac n3+o(n).
\]
\end{theorem}

\begin{proof}
Write the product of the nonabelian Sylow subgroups as
\[
P_1\times\cdots\times P_s,\qquad n_i=\omega(P_i)\ge3.
\]
If $s=0$, then $F$ is abelian and $a(F)=1$, so the first assertion is immediate (and the second is vacuous). Assume henceforth that $s\ge1$. Abelian Sylow factors do not change either $a$ or $\omega$: adjoining an abelian direct factor preserves commutation, and projection onto the nonabelian factors sends every abelian cover to an abelian cover. Taking products of abelian covers gives
\[
a(F)\le\prod_i a(P_i),
\]
while Cartesian products of pairwise noncommuting sets give
\[
n\ge\prod_i n_i.
\]
Hence
\[
\log_2a(F)\le\sum_i\log_2a(P_i).
\]
If $s=1$, apply Theorem~\ref{thm:pgroup}. If $s\ge2$, then for integers $n_i\ge3$,
\[
\sum_i n_i\le\frac23\prod_i n_i\le\frac23n.
\]
Since every coefficient in Theorem~\ref{thm:pgroup} is at most $1/2$,
\[
\log_2a(F)\le\frac12\sum_i n_i+o(n)\le\frac n3+o(n).
\]
For uniformity of the error, write Theorem~\ref{thm:pgroup} as
\[
\log_2a(P_i)\le\frac12n_i+\varepsilon(n_i)n_i+C,
\]
where $\varepsilon(t)\to0$ uniformly in the prime. Fix $R$. Factors with $n_i>R$ contribute at most $\sup_{t>R}\varepsilon(t)\sum_i n_i$, while factors with $n_i\le R$ contribute $O_R(s)$. Since $s\le\log_3n$ and $\sum_i n_i\le n$, first let $n\to\infty$ and then $R\to\infty$.
\end{proof}

\begin{remark}\label{rem:nilpotent-uniform}
The error in Theorem~\ref{thm:pgroup} is uniform in the prime. Consequently the $o(n)$ in Theorem~\ref{thm:nilpotent} is uniform over all finite nilpotent groups: there is a function $\varepsilon(n)\to0$ such that the error is at most $\varepsilon(n)n+O(1)$.
\end{remark}

\section{Extension from a nilpotent normal subgroup}

It remains to recover an arbitrary finite group from a large nilpotent normal subgroup. We use the centralizer of the derived subgroup. The resulting extension cost is quasipolynomial rather than polynomial, but remains negligible on the scale of the main theorem.

\begin{lemma}[Coset domination]\label{lem:coset-domination}
Let $F\trianglelefteq G$, let $X=gF$, and suppose every conjugacy class in $G$ has size at most $B$. Put
\[
Z_0=F\cap Z(G),\qquad M=[F:Z_0].
\]
Then
\[
a(X)\le B(1+\ln M)a(F),
\]
where $a(X)$ denotes the least number of abelian subgroups of $G$ whose union contains $X$.
\end{lemma}

\begin{proof}
Consider the commuting graph on $X/Z_0$. For $x\in X$, its closed neighbourhood is $xC_F(x)/Z_0$, whose size is at least $M/B$, because
\[
[F:C_F(x)]\le[G:C_G(x)]\le B.
\]
A graph on $M$ vertices with minimum closed-neighbourhood size at least $M/B$ has a dominating set of size at most $B(1+\ln M)$. Indeed, if $\delta$ is the minimum degree, select each vertex independently with probability
\[
q=\min\left\{1,\frac{\ln(\delta+1)}{\delta+1}\right\}
\]
and then add every vertex not dominated by the selected set. The expected resulting size is at most
\[
Mq+M(1-q)^{\delta+1}\le\frac{M}{\delta+1}(1+\ln(\delta+1))\le B(1+\ln M).
\]
Thus such a dominating set $D\subseteq X$ exists, and
\[
X=\bigcup_{x\in D}xC_F(x).
\]
Take an abelian cover $F=\bigcup_{i=1}^{a(F)}A_i$. For each $x\in D$,
\[
C_F(x)=\bigcup_i(A_i\cap C_F(x)).
\]
Since $x$ centralizes $A_i\cap C_F(x)$, the subgroup $\langle x,A_i\cap C_F(x)\rangle$ is abelian and contains $x(A_i\cap C_F(x))$. Hence each $xC_F(x)$ is covered by at most $a(F)$ abelian subgroups.
\end{proof}

The preceding lemma applies uniformly to every coset of a normal subgroup. We use it with the centralizer of the derived subgroup; this avoids any appeal to commuting-probability bounds or to the Classification of Finite Simple Groups.

\begin{theorem}[Reduction to a nilpotent normal subgroup]\label{thm:reduction}
Let $G$ be a finite group with $N=\omega(G)$, and put
\[
H=C_G(G').
\]
Then $H\trianglelefteq G$ is nilpotent of class at most two, and
\[
\log_2a(G)\le\log_2a(H)+O\!\left((\log(N+2))^4\right).
\]
\end{theorem}

\begin{proof}
Since $G'$ is characteristic, $H\trianglelefteq G$. Moreover $H'\le G'$, while $H$ centralizes $G'$. Hence $H'\le Z(H)$, so $H$ is nilpotent of class at most two.

Conjugation on $G'$ induces an embedding
\[
G/H\hookrightarrow\Aut(G').
\]
If $m=|G'|$, a finite group of order $m$ has a generating set of size at most $\log_2m$, and an automorphism is determined by the images of such a generating set. Thus
\[
|\Aut(G')|\le m^{\log_2m}.
\]
By Corollary~\ref{cor:derived}, $\log_2m=O((\log(N+2))^2)$, whence
\[
\log_2[G:H]=O((\log(N+2))^4).
\]
Apply Lemma~\ref{lem:coset-domination} to each $H$-coset. By Lemma~\ref{lem:bfc}, every conjugacy class has size at most $B=(2N+1)^2$. For $Z_0=H\cap Z(G)$, Pyber's theorem \cite{Pyber1987} gives
\[
M=[H:Z_0]\le[G:Z(G)]\le C^N
\]
for an absolute constant $C$. Hence $B(1+\ln M)=N^{O(1)}$. Therefore
\[
a(G)\le[G:H]N^{O(1)}a(H),
\]
and taking binary logarithms proves the theorem.
\end{proof}

\section{Completion of the proof}

All ingredients are now in place. The finite reduction confines the problem to finite groups; the nilpotent estimates control $H=C_G(G')$; and the preceding section shows that the extension from $H$ to the whole group costs only $2^{O((\log N)^4)}$. Combining these estimates with the extraspecial construction gives matching exponential upper and lower bounds.

\begin{proof}[Proof of Theorem~\ref{thm:main}]
By Lemma~\ref{lem:finite-reduction}, it suffices to consider finite groups. Let $G$ be finite, put $N=\omega(G)$, and let $H=C_G(G')$. By Theorem~\ref{thm:reduction}, $H$ is nilpotent and
\[
\log_2a(G)\le\log_2a(H)+O((\log(N+2))^4).
\]
We now distinguish the possible nonabelian Sylow structure of $H$.

If $H$ is abelian, then $a(H)=1$, and the required estimate is immediate. If $H$ has a unique nonabelian Sylow subgroup and it is a $2$-group, Proposition~\ref{prop:2group} gives
\[
\log_2a(H)\le\frac{\omega(H)}2+O\!\left(\sqrt N\,(\log(N+2))^3\right)
\le\frac N2+O\!\left(\sqrt N\,(\log(N+2))^3\right).
\]
If the unique nonabelian Sylow subgroup has odd order, Theorem~\ref{thm:pgroup} and Remark~\ref{rem:alpha} give a coefficient uniformly below $1/2$ over odd primes, while the error is $O(N/\log(N+2))$; for sufficiently large $N$ this leaves linear slack, and hence the same displayed upper bound follows. If $H$ has at least two nonabelian Sylow subgroups, the second part of Theorem~\ref{thm:nilpotent} gives
\[
\log_2a(H)\le\frac{\omega(H)}3+o(\omega(H)),
\]
again leaving linear slack. Thus in every case
\[
\log_2a(H)\le\frac N2+O\!\left(\sqrt N\,(\log(N+2))^3\right).
\]
The quasipolynomial extension term $O((\log(N+2))^4)$ is absorbed by the displayed error. Taking the supremum over all $G$ with $\omega(G)\le n$ gives the upper bound. Lemma~\ref{lem:extraspecial} gives
\[
\log_2h(n)\ge\frac n2-O(1),
\]
completing the proof.
\end{proof}

\section{Immediate consequences}

The main theorem has two short consequences that isolate complementary meanings of the same exponential constant.

\begin{corollary}[Least index of an abelian subgroup]\label{cor:index}
For a group $G$, put
\[
\iota(G)=\min\{[G:A]:A\le G\text{ abelian}\},
\]
and define
\[
i(n)=\sup\{\iota(G):\omega(G)\le n\}.
\]
Then, as $n\to\infty$,
\[
\log_2i(n)=\frac n2+o(n),
\]
equivalently $i(n)^{1/n}\to\sqrt2$.
\end{corollary}

\begin{proof}
Neumann's covering lemma \cite{Neumann1954} states that if a group is covered by $k$ cosets of subgroups, then at least one of those subgroups has index at most $k$. An abelian cover of $G$ by $a(G)$ subgroups therefore contains an abelian subgroup $A$ with
\[
[G:A]\le a(G).
\]
Hence $i(n)\le h(n)$ and Theorem~\ref{thm:main} gives $\log_2i(n)\le n/2+o(n)$.

For the reverse inequality, let $E_m$ be an extraspecial $2$-group of order $2^{1+2m}$. If $A\le E_m$ is abelian, then $AZ(E_m)/Z(E_m)$ is isotropic in the $2m$-dimensional symplectic space $E_m/Z(E_m)$, and therefore has dimension at most $m$. Consequently $|A|\le2^{m+1}$ and $[E_m:A]\ge2^m$. By Lemma~\ref{lem:extraspecial}, $\omega(E_m)=2m+1$. Taking $m=\lfloor(n-1)/2\rfloor$ gives
\[
\log_2i(n)\ge\frac n2-O(1),
\]
which proves the claim.
\end{proof}

\begin{corollary}[Asymptotic extremality is concentrated in 2-groups]\label{cor:extremality}
Let $(G_r)$ be a sequence of finite groups, put $n_r=\omega(G_r)$, and suppose that $n_r\to\infty$ and
\[
\log_2a(G_r)=\frac{n_r}{2}-o(n_r).
\]
Put $H_r=C_{G_r}(G_r')$. Then
\[
\omega(H_r)=n_r-o(n_r).
\]
Moreover, for all sufficiently large $r$, $H_r$ has exactly one nonabelian Sylow subgroup, and that Sylow subgroup is a $2$-group.
\end{corollary}

\begin{proof}
By Theorem~\ref{thm:reduction},
\[
\log_2a(G_r)\le\log_2a(H_r)+O((\log(n_r+2))^4).
\]
Since $H_r$ is nilpotent, Theorem~\ref{thm:nilpotent} implies
\[
\log_2a(H_r)\le\frac{\omega(H_r)}2+o(\omega(H_r));
\]
hence asymptotic extremality forces $\omega(H_r)=n_r-o(n_r)$.

If $H_r$ had at least two nonabelian Sylow subgroups along an infinite subsequence, the second part of Theorem~\ref{thm:nilpotent} would give
\[
\log_2a(H_r)\le\frac{\omega(H_r)}3+o(\omega(H_r)),
\]
a contradiction. If $H_r$ were abelian along an infinite subsequence, then $a(H_r)=1$ and Theorem~\ref{thm:reduction} would give $\log_2a(G_r)=O((\log(n_r+2))^4)$, again a contradiction. Thus $H_r$ has exactly one nonabelian Sylow subgroup eventually.

Let it be a $p_r$-group $P_r$. Abelian Sylow factors do not change either $a$ or $\omega$. If $p_r$ were odd along an infinite subsequence, Theorem~\ref{thm:pgroup} and Remark~\ref{rem:alpha} would give a uniform coefficient strictly below $1/2$, again a contradiction. Hence $p_r=2$ eventually.
\end{proof}

\begin{remark}[Chromatic-versus-clique formulation]
Equivalently, for the noncommuting graphs themselves,
\[
\sup\{\chi(\Gamma_G):\omega(\Gamma_G)\le n\}=2^{n/2+o(n)}.
\]
Thus the theorem determines the sharp asymptotic chromatic-versus-clique growth within the class of noncommuting graphs. We do not use the term ``$\chi$-bounded class'' in its standard hereditary sense, because an induced subgraph of a noncommuting graph need not itself be the noncommuting graph of a group.
\end{remark}

\section*{Statements and Declarations}

\noindent\textbf{Funding.} No funds, grants, or other support were received during the preparation of this manuscript.

\medskip
\noindent\textbf{Competing interests.} The author has no relevant financial or non-financial interests to disclose.

\medskip
\noindent\textbf{Data availability.} No datasets were generated or analysed during the current study.

\medskip
\noindent\textbf{Use of artificial intelligence.} Generative artificial-intelligence tools were used as auxiliary tools for proof checking, consistency checks, and \TeX{} editing. All mathematical statements, arguments, citations, and the final manuscript were reviewed and remain the sole responsibility of the author.

\end{document}